\documentclass[reqno]{amsart}
\usepackage{tabu}
\usepackage{amssymb}
\usepackage{cancel}
\usepackage{mathtools}
\usepackage{a4wide,amsmath}
\usepackage{mathrsfs}
\usepackage{amsthm}
\usepackage{dsfont}
\numberwithin{equation}{section}
\numberwithin{figure}{section}
\numberwithin{table}{section}
\usepackage{bbm}
\usepackage{subfig}
\usepackage{enumerate}
\usepackage[section]{placeins}
\usepackage{graphicx}		  % Tilføjelse af billedfiler
\usepackage{ifpdf}
\usepackage{comment}
\ifpdf
\DeclareGraphicsExtensions{.pdf,.eps,.jpg,.png}	
\usepackage[suffix=]{epstopdf}
\fi
\usepackage{xcolor}
\usepackage[utf8]{inputenc}
\usepackage{hyperref}
\hypersetup{hidelinks}

\mathtoolsset{showonlyrefs}

\long\def\MSC#1\EndMSC{\def\arg{#1}\ifx\arg\empty\relax\else
	{\narrower\noindent%
		{2020 Mathematics Subject Classification}: #1\\} \fi}
\long\def\PACS#1\EndPACS{\def\arg{#1}\ifx\arg\empty\relax\else
	{\narrower\noindent%
		{PACS numbers}: #1}\fi}
\long\def\KEY#1\EndKEY{\def\arg{#1}\ifx\arg\empty\relax\else
	{\narrower\noindent% 
		Keywords: #1\\}\fi}

\usepackage{listings}
\lstnewenvironment{mat}
{\lstset{language=mathematica,mathescape,columns=flexible}}
{}

\newcommand{\XXint}[3]{{\setbox0=\hbox{$#1{#2#3}{\int}$}
		\vcenter{\hbox{$#2#3$}}\kern-.5\wd0}}

\theoremstyle{plain}
\newtheorem{theorem}{Theorem}[section]
\newtheorem{lemma}[theorem]{Lemma}
\newtheorem{proposition}[theorem]{Proposition}
\newtheorem{corollary}[theorem]{Corollary}
\theoremstyle{definition}

\theoremstyle{remark}
\newtheorem{remark}[theorem]{Remark}
\newcommand{\abs}[1]{{\lvert#1\rvert}} 
\newcommand{\inner}[1]{\langle#1\rangle}

\newcommand{\essinf}{\mathop{\textup{ess\,inf}}}

\newcommand{\vek}[1]{\mathbf{#1}}

\newcommand{\I}{\mathrm{i}}    % imaginary unit
\newcommand{\e}{\mathrm{e}}    % Euler's number
\newcommand{\di}{\mathrm{d}}   % differential

\newcommand{\R}{\mathbb{R}}
\newcommand{\N}{\mathbb{N}}
\newcommand{\C}{\mathbb{C}}
\newcommand{\Z}{\mathbb{Z}}

\begin{document}
	\title[On the linearized forward map of EIT]{Continuity of the linearized forward map of Electrical Impedance Tomography from square-integrable perturbations to Hilbert--Schmidt operators}
    
\author[J.~Bisch]{Joanna Bisch}
\address[J.~Bisch]{Department of Mathematics and Systems Analysis, Aalto University, P.O. Box~11100, 00076 Helsinki, Finland.}
\email{joanna.bisch@aalto.fi}

\author[M.~Hirvensalo]{Markus Hirvensalo}
\address[M.~Hirvensalo]{Department of Mathematics and Systems Analysis, Aalto University, P.O. Box~11100, 00076 Helsinki, Finland.}
\email{markus.hirvensalo@aalto.fi}

\author[N.~Hyv\"onen]{Nuutti Hyv\"onen}
\address[N.~Hyv\"onen]{Department of Mathematics and Systems Analysis, Aalto University, P.O. Box~11100, 00076 Helsinki, Finland.}
\email{nuutti.hyvonen@aalto.fi}
	
\begin{abstract}
This work considers the Fr\'echet derivative of the idealized forward map of two-dimensional electrical impedance tomography,~i.e.,~the linear operator that maps a perturbation of the coefficient in the conductivity equation over a bounded two-dimensional domain to the linear approximation of the corresponding change in the Neumann-to-Dirichlet boundary map. It is proved that the Fr\'echet derivative is bounded from the space of square-integrable conductivity perturbations to the space of Hilbert--Schmidt operators on the mean-free $L^2$ functions on the domain boundary, if the background conductivity coefficient is constant and the considered simply-connected domain has a $C^{1,\alpha}$ boundary. This result provides a theoretical framework for analyzing linearization-based one-step  reconstruction algorithms of electrical impedance tomography in an infinite-dimensional setting.
\end{abstract}	
\maketitle
	
\KEY
electrical impedance tomography,
conductivity equation,
linearization,
Hilbert--Schmidt operators,
Calder\'on problem, 
Zernike polynomials.
\EndKEY
\MSC
35B30, 35B35, 35Q60, 35R30.
\EndMSC

\section{Introduction}
The goal of {\em electrical impedance tomography} (EIT) is to reconstruct the internal conductivity of an examined body $\Omega$ from boundary measurements of current and voltage. According to the idealized {\em continuum model} (CM), the boundary data attainable by EIT is the {\em Neumann-to-Dirichlet} (ND), or the {\em Dirichlet-to-Neumann} (DN), boundary map for the conductivity equation
$$
-\nabla \cdot (\gamma \nabla u) = 0 \qquad \text{in } \Omega,
$$
where the positive coefficient $\gamma \in L^\infty(\Omega)$ is the to-be-reconstructed conductivity in $\Omega$. Although the DN map is preferred in many theoretical works on EIT, we resort here to the ND map due to its more favorable numerical properties. For more information on EIT, including the unique identifiability of $\gamma$ from boundary measurements, (in)stability estimates and basics of reconstruction algorithms, we refer to the review papers~\cite{Borcea02, Cheney99, Uhlmann09} and the references therein.

This work is motivated by the simplest approach to reconstructing useful information on the conductivity from boundary measurements modeled by the CM: linearizing the forward map that sends $\gamma$ to the ND operator at some constant conductivity level and solving the resulting linearized inverse problem via,~e.g.,~regularization \cite{Engl96} or Bayesian inversion \cite{Kaipio04a,Stuart10}. It is well-known that the forward map of EIT is Fr\'echet differentiable, with the standard version of the derivative mapping $L^\infty(\Omega)$ to the space of bounded linear operators on, say, the mean-free subspace $L^2_\diamond(\partial \Omega)$ of $L^2(\partial \Omega)$. In particular, the natural domain and image spaces for the Fr\'echet derivative are not Hilbert spaces, and their duals are also rather unpleasant objects from the standpoint of numerical algorithms. This complicates the theory for solving the linearized inverse problem of EIT since both regularization and Bayesian techniques work most naturally for Hilbert spaces, and they often explicitly utilize the adjoint operator of the linear(ized) forward map.

Despite the aforementioned theoretical complications in the infinite-dimensional setting, algorithms based on one-step linearization have been successfully applied to solving the discretized reconstruction problem of EIT in practice; see,~e.g.,~\cite{Adler09, Cheney90} for such approaches in the context of realistic electrode measurements. A computational framework can be introduced, e.g., as follows: After discretizing the conductivity and choosing a finite-dimensional $L^2(\partial \Omega)$-orthonormal  basis for the boundary measurements, one can introduce a finite-dimensional version of the Fr\'echet derivative that maps the discretized conductivity perturbation to a truncated matrix representation in the chosen basis for the corresponding change in the ND map. This can be achieved,~e.g.,~with the help of some finite element method. However, the connection between such a finite-dimensional computational setting and the infinite-dimensional linearization of the CM is nonobvious: If the boundary data matrix is vectorized and the discrete Fr\'echet derivative is interpreted as a linear mapping between Euclidean spaces, one would expect that the domain and image spaces encountered at the discretization limit,~i.e.,~when the conductivity discretization gets infinitely fine and the dimension of the boundary current basis approaches infinity, are not $L^\infty(\Omega)$ and the bounded linear maps on $L^2_\diamond(\partial \Omega)$, but rather (weighted) $L^2(\Omega)$ and the space of Hilbert--Schmidt operators on $L^2_\diamond(\partial \Omega)$. 

The main result of this work is that the Fr\'echet derivative of the idealized forward map of EIT evaluated at a constant conductivity is compatible with the heuristic discretization limit considered above, that is, it maps $L^2(\Omega)$ continuously to the space of Hilbert--Schmidt operators on $L^2_\diamond(\partial \Omega)$, if $\Omega$ is a bounded simply-connected two-dimensional $C^{1, \alpha}$ domain. This provides an infinite-dimensional Hilbert space framework for further analysis of linearization-based one-step reconstruction algorithms for EIT.

The fact that the Fr\'echet derivative for the forward map of the CM evaluated at any positive Lipschitz conductivity extends in two dimensions to a bounded map from $L^2(\Omega)$ to the space of bounded linear operators on $L^2_\diamond(\partial \Omega)$ was established in \cite{2d_paperi}. However, the question on whether the boundedness of the derivative is retained on $L^2(\Omega)$ when switching on the image side to the Hilbert--Schmidt topology was left open by \cite{2d_paperi}, although it did prove such a result for certain infinite-dimensional subspaces of $L^2(\Omega)$. On the other hand, it seems to be common knowledge (cf.,~e.g.,~\cite{Hanke03}) that the ND map is a Hilbert--Schmidt operator for positive $L^\infty(\Omega)$ conductivities in regular enough two-dimensional domains; see \cite[Appendix~A]{Garde2022} for a formal proof of this result.

This text is organized as follows. Section~\ref{sec:setting} introduces the problem setting and states our main theorem and a corollary that considers the possibility to numerically approximate the Fr\'echet derivative. The proof of the main theorem for the case that $\Omega$ is the unit disk is divided over Sections~\ref{sec:F-matrices}--\ref{sec:schur}. Finally, Section~\ref{sec:generalization} extends the argumentation for more general two-dimensional simply-connected domains.

\subsection{On the notation}
The space of bounded linear operators between Banach spaces $X$ and $Y$ is denoted by $\mathscr{L}(X, Y)$, with the shorthand notation $\mathscr{L}(X) =  \mathscr{L}(X, X)$. Analogously, the space of Hilbert--Schmidt operators between Hilbert spaces $H_1$ and $H_2$ is denoted by $\mathscr{L}_{\rm HS}(H_1, H_2)$. For more information on Hilbert--Schmidt operators, consult,~e.g.,~\cite{Weidmann1980}. 

\section{Problem setting and main results}
\label{sec:setting}
    Let \( \Omega \subset \R^2\) be a bounded Lipschitz domain whose conductivity is characterized by a real-valued function \( \gamma \in L^\infty(\Omega; \R) \),\footnote{Unless explicitly indicated, all functions spaces in this work have $\C$ as the multiplier field.} with \( \essinf \gamma > 0 \). Denote by $\langle \, \cdot \,, \, \cdot \, \rangle$ the $L^2(\partial \Omega)$ inner product and consider a mean-free boundary current density
    \begin{equation*}
        f \in L^2_\diamond(\partial \Omega) = \big \{ g\in L^2(\partial \Omega) \mid \inner{g,1} = 0 \big \}.
    \end{equation*}
    The electromagnetic potential induced by $f$ weakly satisfies the elliptic problem
    \begin{equation} \label{eq:strong_form}
        -\nabla\cdot(\gamma\nabla u) = 0 \text{ in } \Omega, \qquad \nu\cdot(\gamma\nabla u) = f \text{ on } \partial \Omega,
    \end{equation}
    where \( \nu \) is the exterior unit normal of $\Omega$ and $ \cdot  $ denotes the real dot product.
    The variational formulation of \eqref{eq:strong_form} is to find \( u \) such that
    \begin{equation} \label{eq:weak_form}
        \int_\Omega \gamma \nabla u \cdot \overline{\nabla v} \, \di V = \langle f, v|_{\partial \Omega} \rangle \qquad \forall v \in H^1(\Omega).
    \end{equation}
    With the help of the Lax--Milgram lemma, it straightforwardly follows that there exists a unique solution \( u^\gamma_f \) for \eqref{eq:weak_form} in the Sobolev space
    \begin{equation*}
        H^1_\diamond(\Omega) =  \big\{ w \in H^1(\Omega) \; | \; w|_{\partial \Omega} \in L^2_\diamond(\partial \Omega) \big \}. 
    \end{equation*}
    The dependence between the boundary current density and the boundary potential in \eqref{eq:weak_form} can be described by the linear ND boundary map
    \begin{equation*}
        \Lambda(\gamma): \left\{
        \begin{array}{l}
            f \mapsto  u^\gamma_f |_{\partial \Omega}, \\[2mm]
            L^2_\diamond(\partial \Omega) \to L^2_\diamond(\partial \Omega),
        \end{array}
        \right.
    \end{equation*}
    which is a standard (idealized) input for the inverse problem of determining $\gamma$ from boundary measurements of current and voltage. In two spatial dimensions, the ND map is a Hilbert--Schmidt operator if $\partial \Omega$ is of the class $C^{1, \alpha}$ for some $\alpha > 0$ \cite[Theorem~A.2]{Garde2022}, which is the regularity we assume for $\partial \Omega$ in the following (except for some parts of Section~\ref{sec:generalization}).
    
    It is well-known that the nonlinear {\em forward map}
    \begin{equation}
        \gamma \mapsto \Lambda(\gamma)
    \end{equation}
    is Fr\'echet differentiable with respect to complex-valued perturbations \( \eta \in L^\infty(\Omega) \); see,~e.g.,~\cite{Lechleiter08,Garde2020}. Denote the Fr\'echet derivative of \( \Lambda \) at \( \gamma = 1 \) by \( F = D \Lambda(1) \in \mathscr{L}(L^\infty(\Omega),\mathscr{L}_{\text{HS}}(L^2_\diamond(\partial \Omega))) \) and note that it is uniquely characterized by the identity
    \begin{equation} \label{eq:frechet_redef}
       \big\langle (F \eta) f, g \big\rangle = -\int_\Omega \eta \nabla u_f^1 \cdot \overline{\nabla u_g^1} \,\di V
    \end{equation}
    for all \( \eta \in L^\infty(\Omega) \) and \( f, g \in L^2_\diamond(\partial \Omega) \).
    
 Using \eqref{eq:frechet_redef} as the definition of $F$ and exploiting elliptic regularity theory, $F$ can be extended to an element of $\mathscr{L}(L^2(\Omega),\mathscr{L}(L^2_\diamond(\partial \Omega)))$, i.e., to be bounded from the space of square-integrable conductivity perturbations to the space of bounded linear boundary maps on $L^2_\diamond(\partial \Omega)$ \cite[Proposition~1.1]{2d_paperi}. However, the analysis in \cite{2d_paperi} does not reveal weather $F\eta$ remains a Hilbert--Schmidt operator for a general $\eta \in L^2(\Omega)$ (cf.~\cite[Theorem~1.4]{2d_paperi}), which is the question settled by our main result:  

    \begin{theorem} \label{thm:main}
        Let \( \Omega \subset \R^2\) be a bounded simply-connected \( C^{1,\alpha} \) domain for some \( \alpha \in (0, 1) \).
        Then, the linear map
        \[ 
        F : L^2(\Omega) \longrightarrow \mathscr{L}_{\rm HS}\big( L^2_\diamond(\partial \Omega) \big) 
        \]
        is continuous.
    \end{theorem}

    From the standpoint of numerically approximating $F$, mere boundedness between the Hilbert spaces $L^2(\Omega)$ and $\mathscr{L}_{\rm HS}( L^2_\diamond(\partial \Omega) )$ is not enough, but one also needs compactness that allows approximation by operators of finite rank. However, \cite[Theorem~1.4]{2d_paperi} indicates that
    \[
    \|F \eta \|_{\mathscr{L}_{\rm HS}( L^2_\diamond(\partial \Omega) )} \geq c \| \eta \|_{L^2(\Omega)}
    \]
    for all $\eta$ in certain infinite-dimensional closed subspaces of $L^2(\Omega)$, demonstrating that $F : L^2(\Omega) \to \mathscr{L}_{\rm HS}( L^2_\diamond(\partial \Omega) )$ cannot be compact as such. 

    A straightforward way to introduce a compact version of $F: L^2(\Omega) \to \mathscr{L}_{\rm HS}(L^2_\diamond(\partial \Omega) )$, without losing the attractive Hilbert space structures of its domain and image spaces, is to consider conductivity perturbations in $H^{\varepsilon}(\Omega)$, $\varepsilon > 0$, and exploit the compactness of the embedding $H^{\varepsilon}(\Omega) \hookrightarrow L^2(\Omega)$. Here, we consider a less trivial option and tamper with the image space instead. To this end, let 
    \begin{equation}
    H^{s}_\diamond(\partial \Omega) = \big \{ g\in H^s(\partial \Omega) \mid \inner{g,1}_s = 0 \big \}, \qquad s \in [-\tfrac{1}{2}, \tfrac{1}{2}],
    \end{equation}
    with $\langle \, \cdot \, , \, \cdot \, \rangle_s: H^s(\partial \Omega) \times H^{-s}(\partial \Omega) \to \C$ denoting the {\em sesquilinear} dual evaluation between $H^s(\partial \Omega)$ and $H^{-s}(\partial \Omega)$, which can be understood as a generalization of the $L^2(\partial \Omega)$ inner product.

    \begin{corollary} \label{cor:1}
        Let $\Omega$ be as in Theorem~\ref{thm:main} and
         \( 0 < \varepsilon \leq \tfrac{1}{2} \). 
        Then, the linear map
        \[ F : L^2(\Omega) \longrightarrow \mathscr{L}_{\rm HS}\left( H^{\varepsilon}_\diamond(\partial \Omega), H^{-\varepsilon}_\diamond(\partial \Omega) \right) \]
        is compact.
    \end{corollary}

The proofs of Theorem~\ref{thm:main} and Corollary~\ref{cor:1} are structured as follows. By drawing on material in \cite{2d_paperi} and \cite[Section~3]{Autio2024}, we first demonstrate in Section~\ref{sec:F-matrices} that the continuity of $F : L^2(\Omega) \to \mathscr{L}_{\rm HS}( L^2_\diamond(\partial \Omega) )$ can be reduced to the uniform boundedness of a certain countable set of linear operators on $\ell^2(\N)$, represented as infinite matrices, if $\Omega$ is the unit disk $D$. Section~\ref{sec:asymptotic} utilizes Gr\"onvall's inequality to establish upper bounds for the elements of these matrices, which facilitates their treatment as certain complicated product terms are replaced by simple exponential expressions. The proof of Theorem~\ref{thm:main} for $\Omega = D$ is then completed in Section~\ref{sec:schur} by resorting to the Schur test. Finally, Section~\ref{sec:generalization} proves first Theorem~\ref{thm:main} in its general form by employing the Riemann mapping theorem (cf.~\cite{2d_paperi}) and then Corollary~\ref{cor:1} by showing that the embedding $\mathscr{L}_{\rm HS}( L^2_\diamond(\partial \Omega)) \hookrightarrow\mathscr{L}_{\rm HS}( H^{\varepsilon}_\diamond(\partial \Omega), H^{-\varepsilon}_\diamond(\partial \Omega))$ is compact for $0 < \varepsilon \leq \tfrac{1}{2}$ if $\Omega$ is (only) Lipschitz.
 
\section{Infinite matrix representation for the operator \texorpdfstring{\( F \)}{F} in the unit disk}
\label{sec:F-matrices}

Let us assume that $\Omega = D$ is the unit disk. Following the ideas in \cite{2d_paperi,Autio2024}, we introduce an orthonormal Zernike polynomial basis \cite{Zernike1934} for $L^2(D)$ in the polar coordinates $(r,\theta)$ via 
\begin{equation} \label{eq:zernike}
    \psi_{j, k}(r, \theta) = \sqrt{\frac{\abs{j} + 2 k + 1}{\pi}} R^{\abs{j}}_{\abs{j} + 2 k}(r) \, \e^{\I j\theta}, \qquad j \in \Z, \ k \in \N_0,
\end{equation}
where 
\begin{equation*}
    R^{\abs{j}}_{\abs{j} + 2 k}(r) = \sum^k_{i=0} (-1)^i \binom{\abs{j} + 2 k - i}{i} \binom{\abs{j} + 2 k - 2 i}{k - i} \, r^{\abs{j} + 2 k - 2 i}.
\end{equation*}
A given $\eta \in L^2(D)$ is expanded in the Zernike basis as
\[
\eta = \sum_{j \in \Z} \sum_{k \in \N_0} c_{j,k}(\eta) \psi_{j,k}
\]
where $c_{j,k}(\eta) = \langle \eta, \psi_{j,k}\rangle_{L^2(D)}$. The standard Fourier basis (without the constant function)
  \begin{equation*}
        f_m(\theta) = \frac{1}{\sqrt{2 \pi}}\e^{\I m\theta}, \qquad m \in \Z' := \Z \setminus \{ 0 \},
    \end{equation*}
    serves as our orthonormal basis for $L^2_\diamond(\partial \Omega)$. According to \cite[Eq.~(4.5)]{2d_paperi}, these bases interplay with the linearized forward map $F$ as follows:
        \begin{equation} \label{eq:expl_coef}
        a^{j,k}_{m,n} = \big \langle (F \psi_{j,k})f_m, f_n \big \rangle = - \frac{1}{\sqrt{\pi}} \, \frac{\sqrt{\abs{j} + 2 k + 1}}{\min \{\abs{m},\abs{n} \} + \abs{j} + k} \, \prod^k_{i=1} \frac{\min \{\abs{m},\abs{n} \} - i}{\abs{j} + \min \{\abs{m},\abs{n} \} + k - i}
    \end{equation}
if $n=m+j$, $k < \min \{\abs{m},\abs{n} \}$ and $mn > 0$, and $a^{j,k}_{m,n}= 0$ for all other combinations of $j \in \Z$, $k \in \N_0$ and $m,n \in \Z'$. When $k=0$, the product in \eqref{eq:expl_coef} is defined to take the value $1$. 

    Define
    \[
    \mathcal{K}_j = \overline{\text{span}\{\psi_{j,k} \, | \, k \in \N_0 \}} \quad \text{such that} \quad L^2(D) = \bigoplus_{j \in \Z} \mathcal{K}_j
    \]
    and note that the orthogonal projection $P_j: L^2(D) \to \mathcal{K}_j$ is given by
    \[
    P_j\eta = \sum_{k \in \N_0} c_{j,k}(\eta) \psi_{j,k}.
    \]
    Let us expand 
    \begin{equation}
    \label{eq:P_expansion}
    F = \sum_{j \in \Z} F P_j.
    \end{equation}
    For $\eta \in L^2(D)$, the matrix coefficients of the bounded linear operator $F P_j \eta: L^2_\diamond(\partial D) \to L^2_\diamond(\partial D)$ in the Fourier basis read
    \begin{equation}
    \label{eq:FPj_coef}
    a^j_{m,n}(\eta) = \big \langle (F P_j \eta)f_m, f_n \big \rangle = \sum_{k \in \N_0} a_{m,n}^{j,k} c_{j,k}(\eta), \qquad m, n \in \Z',
    \end{equation}
    which are nonzero only if $n = m + j$ due to \eqref{eq:expl_coef}. This means that for any $\eta \in L^2(D)$, the matrix representation of $F P_j \eta$ in the Fourier basis only has nonzero entries on its $j$th diagonal. In particular,
    \begin{equation}
    \label{eq:HS-ortho}
     \big \langle F P_i \eta_1, F P_j \eta_2 \big \rangle_{\mathscr{L}_{\rm HS}(L^2_\diamond(\partial D))}
    = 0 \qquad \text{if} \ i \not= j
    \end{equation}
    for any $\eta_1, \eta_2 \in L^2(D)$ because $F P_i \eta_1$ and  $F P_j \eta_2$ do not have nonzero elements at same positions in their matrix representations if $i \not= j$.

    Furthermore, it follows from \eqref{eq:expl_coef} that the only nonempty diagonal,~i.e.,~the $j$th one, in the matrix representation \eqref{eq:FPj_coef} for $F P_j$ can be given with the help of an infinite lower a triangular matrix $F^{|j|}$ given component-wise as \cite[eq.~(3.11) \& Remark~3.1]{Autio2024},
    \begin{equation} \label{eq:inf_F_matrix}
         F^{\abs{j}}_{m,k} = \left\{
        \begin{array}{ll}
            a^{\abs{j},k-1}_{m,m+\abs{j}} \quad & \text{if } 1 \leq k  \leq m \in \N, \\[1mm]
            0 \quad & \text{otherwise},
        \end{array}
        \right.
    \end{equation} 
    and the vectorized Zernike coefficients for the angular frequency $j$
    \[
    \vek{c}^j = \big[c_{j,k-1}(\eta)\big]_{k=1}^\infty.
    \]
    More precisely, 
    \begin{equation}
    \label{eq:a_diag}
     a^j_{m,m+j}(\eta) = 
    \left\{
    \begin{array}{ll}
    a^{j}_{-m-j,-m}(\eta), & \quad m < \min\{-j,0\}, \\[1mm]
    0,  & \quad \min\{-j,0\} < m < \max\{-j,0\}, \\[1mm]
    (F^{|j|} \vek{c}^j)_{\min\{m,m+j\}}, & \quad m > \max\{-j,0\}
    \end{array}
    \right.
    \end{equation}
for any $j \in \Z$. Note that the empty quadrants in \eqref{eq:a_diag} and the triangular structure of $F^{\abs{j}}$ originate from the conditions $m n > 0$ and  $k < \min \{\abs{m},\abs{n} \}$, respectively, for the nonzero elements in \eqref{eq:expl_coef}; see \cite[(3.8)--(3.10) \& Remark~3.1]{Autio2024}. In particular,
\begin{equation}
\label{eq:HS_l2}
    \| FP_j \eta \|_{\mathscr{L}_{\rm HS}(L^2_\diamond(\partial D))}^2 = \sum_{m,n \in \Z'} |a^j_{m,n}(\eta)|^2  = 
    \sum_{m \in \Z'} |a^j_{m,m+j}(\eta)|^2 = 2 \big\| F^{|j|} \mathbf{c}^j \big\|_{\ell^2(\N)}^2,
    \end{equation}
    where the final equality holds due to the symmetry of the presentation for the nonempty diagonal of the infinite matrix $(a^j_{m,n})_{m,n \in \Z'}$ in \eqref{eq:a_diag} with respect to the (possibly virtual) zero element $a^j_{-j/2,j/2}(\eta)$. Combined with \eqref{eq:HS-ortho}, this provides the interface for proving the sought-for connection between the Hilbert--Schmidt norm of $F \eta$ and the norms of $F^{|j|}$ on $\ell^2(\N)$.
    
    \begin{lemma} \label{lemma:norms}
        For any \( \eta \in L^2(D) \),
        \begin{equation}
            \| F \eta \|_{\mathscr{L}_{\rm HS}(L^2_\diamond(\partial D))} \leq \sqrt{2} \, \sup_{j \in \Z} \| F^\abs{j} \|_{\mathscr{L}( \ell^2(\N))} \, \| \eta \|_{L^2(D)}.
        \end{equation}
    \end{lemma}
    \begin{proof}
        By virtue of \eqref{eq:P_expansion}, \eqref{eq:HS-ortho} and \eqref{eq:HS_l2}, we have
        \begin{equation*}
            \| F\eta \|^2_{\mathscr{L}_\text{HS}(L^2_\diamond(\partial D))} 
            = \sum_{j \in \Z} \| FP_j \eta \|_{\mathscr{L}_\text{HS}(L^2_\diamond(\partial D))}^2 = 2\sum_{j \in \Z}  \big\| F^{|j|} \mathbf{c}^j \big\|_{\ell^2(\N)}^2.
        \end{equation*}
        Hence,
        \[
        \| F\eta \|^2_{\mathscr{L}_\text{HS}(L^2_\diamond(\partial D))} \leq 2 \sup_{j \in \Z} \| F^\abs{j} \|_{\mathscr{L}( \ell^2(\N))}^2 \sum_{j \in \Z}  \| \mathbf{c}^j \|_{\ell^2(\N)}^2 =  2 \sup_{j \in \Z} \| F^\abs{j} \|_{\mathscr{L}( \ell^2(\N))}^2 \sum_{j \in \Z} \sum_{k \in \N_0}  | c_{j,k}(\eta) |^2.
        \]
        Since $\{ c_{j,k}(\eta)\}_{j \in \Z, k \in \N_0}$ are the coefficients of an expansion of $\eta$ with respect to an orthonormal basis of $L^2(D)$, the proof is complete.
    \end{proof}

\section{Upper bounds for the elements of \texorpdfstring{\( F \)}{F}}
\label{sec:asymptotic}

A special feature of $F^{|j|}$, $j \in \Z$, is that all its elements are nonpositive, which seems compatible  with using the classical Schur test for proving the boundedness of $F^{|j|}$. However, the product term in \eqref{eq:expl_coef} leads to technical difficulties in directly applying such a strategy, and thus the purpose of this section is to bound the product term by an exponential expression. This enables using the integral test in connection to the Schur test in Section~\ref{sec:schur}.

    From \eqref{eq:inf_F_matrix} and \eqref{eq:expl_coef}, we get
    \begin{align} \label{eq:abs_F}
        \big| F_{m,k}^{\abs{j}} \big| &= \frac{1}{\sqrt{\pi}} \frac{\sqrt{2k+\abs{j}-1}}{m+\abs{j}+k-1}\prod_{i=1}^{k-1}\frac{m-i}{m+\abs{j}+k-1-i} \nonumber \\
        &= \frac{1}{\sqrt{\pi}} \frac{\sqrt{2k+\abs{j}-1}}{m+\abs{j}+k-1} \frac{\Gamma(m) \Gamma(m + \abs{j})}{\Gamma(m - k + 1) \Gamma(m + \abs{j} + k - 1)} \nonumber \\[1mm]
        &= \frac{1}{\sqrt{\pi}} \frac{\sqrt{2k+\abs{j}-1}}{m+\abs{j}} \frac{\Gamma(m) \Gamma(m + \abs{j} + 1)}{\Gamma(m - k + 1) \Gamma(m + \abs{j} + k)}
    \end{align}
    for $1 \leq k \leq m \in \N$ and $j \in \Z$. Here, \( \Gamma \) denotes the gamma function and we used the identity \( z \Gamma(z) = \Gamma(z + 1) \) in the simplification.
    Let us isolate the fraction of gamma functions in \eqref{eq:abs_F} and replace \( k \) with a continuous variable~\( x \):
    \begin{equation} \label{eq:rho}
        \rho(x) := \frac{ \Gamma(m) \, \Gamma(m + \abs{j} + 1)}{\Gamma(m - x + 1) \, \Gamma(m + \abs{j} + x)},\qquad 1 \leq x \leq m.
    \end{equation}
    We are interested in finding an upper bound for \eqref{eq:rho} with respect to \( x \) while keeping \( m \text{ and } j \) fixed.

    One could consider,~e.g.,~Stirling's approximation for the gamma function, but for our purposes a different approach turns out more productive.
    Let us differentiate \eqref{eq:rho} with respect to \( x \). After using the product rule and the identity
    \begin{equation*}
        \frac{\di}{\di z} \frac{1}{\Gamma(z)} = - \frac{\psi(z)}{\Gamma(z)},
    \end{equation*}
    where \( \psi(z) \) is the digamma function, we get
    \begin{equation}\label{eq:D_rho}
        \frac{\di}{\di x} \rho(x) = \rho(x) \big( \psi(m - x + 1) - \psi(m + \abs{j} + x) \big).
    \end{equation}
    This means that the ratio between the derivative of \( \rho(x) \) and \( \rho(x) \) itself, i.e., the logarithmic derivative of $\rho(x)$, is a difference of two shifted digamma functions \( \psi(m - x + 1)-\psi(m + \abs{j} + x) \), which is negative for all \( 1 \leq x \leq m \) due to the strict monotonicity of the digamma function on the positive real axis. Morally, if we replace this difference by something less negative in \eqref{eq:D_rho}, we can construct a function that has a lower rate of decay than $\rho$ by resorting to Gr\"onwall's inequality.

    The derivative of the digamma function is the trigamma function \( \psi_{1}(z) \) that is a strictly convex and strictly decreasing function on the positive real axis, admitting the lower bound (see,~e.g.,~\cite[Lemma~1]{Guo2013}),
    \begin{equation} \label{eq:trigamma_ineq}
        \psi_1(z) \geq \frac{1}{z}, \qquad  z > 0.
    \end{equation}
    We may thus use the fundamental theorem of calculus, the Jensen's inequality and \eqref{eq:trigamma_ineq} to get an estimate for the difference of digamma functions in \eqref{eq:D_rho}:
    \begin{align*}
        - \left( \psi(m + \abs{j} + x) - \psi(m - x + 1) \right)
        &= - (2 x + \abs{j} - 1) \left( \frac{1}{2 x + \abs{j} - 1} \int_{m - x + 1}^{m + \abs{j} + x} \psi_1(z) \, \di z \right) \\[1mm]
        &\leq - (2 x + \abs{j} - 1) \, \psi_1 \left( \frac{1}{2 x + \abs{j} - 1} \int_{m - x + 1}^{m + \abs{j} + x} z \, \di z \right) \\[1mm]
        &\leq - \frac{2 x + \abs{j} - 1}{m + (\abs{j} + 1) / 2}, \qquad 1 \leq x \leq m.
    \end{align*}
    As $\rho(x) > 0$ for $1 \leq x \leq m$, it follows from \eqref{eq:D_rho} that
    \begin{equation}\label{eq:D_rho_ine}
        \frac{\di}{\di x } \rho(x) \leq \rho(x) \! \left( - \frac{2 x + \abs{j} - 1}{m + (\abs{j} + 1) / 2} \right), \qquad 1 \leq x \leq m.
    \end{equation}
   Applying the differential form of Gr\"onwall's inequality to \eqref{eq:D_rho_ine} gives
    \begin{equation} \label{eq:rho_ine}
        \rho(x) 
        \leq \rho(1) \exp \left( - \int_1^x \frac{2 z + \abs{j} - 1}{m + (\abs{j} + 1) / 2}  \di z \right) 
        = {\rm e}^{ - \frac{2 (x + \abs{j}) (x - 1)}{2 m + \abs{j} + 1}},  \qquad 1 \leq x \leq m.
    \end{equation}
    Substituting the estimate \eqref{eq:rho_ine} for $\rho$ in \eqref{eq:abs_F} finally provides the desired upper bounds for the absolute values of the matrix elements,
    \begin{equation} \label{eq:asymp}
        \big| F_{m, k}^{\abs{j}} \big| \leq \frac{1}{\sqrt{\pi}} \frac{\sqrt{2 k + \abs{j} - 1}}{m + \abs{j}} \, {\rm e}^{-\frac{2(k + |j|) (k - 1)}{2 m + \abs{j} + 1 }} =: \xi^\abs{j}_m(k), \qquad 1 \leq k \leq m < \infty, \ \ j \in \Z,
    \end{equation}
    that according to our numerical tests seem to capture the asymptotic behavior of \eqref{eq:abs_F} when $m$ and/or $k$ approach infinity.
    
    Indeed, let us numerically illustrate the sharpness of the upper bound in \eqref{eq:asymp}. 
    Figure~\ref{fiq:upper} compares $| F_{m, k}^{\abs{j}} |$ and \( \xi^\abs{j}_{m}(x) \), with the understanding that $x$ is a continuum version of $1 \leq k \leq m$.
    We have selected the indices \( m \in \{ 15, 30, 100 \} \) and \( j \in \{ 0, 3 \} \), for which $| F_{m, k}^{\abs{j}} |$ and \( \xi^\abs{j}_{m}(x) \) are plotted as functions \( k \) and \( x \), respectively, over the line segment \( [1, 16] \).
    Based on this visual demonstration, the upper bound \eqref{eq:asymp} seems reasonable.
    \begin{figure}[ht]
    \centering
    \begin{minipage}{0.5\textwidth}
        \centering
        \includegraphics[width=\textwidth]{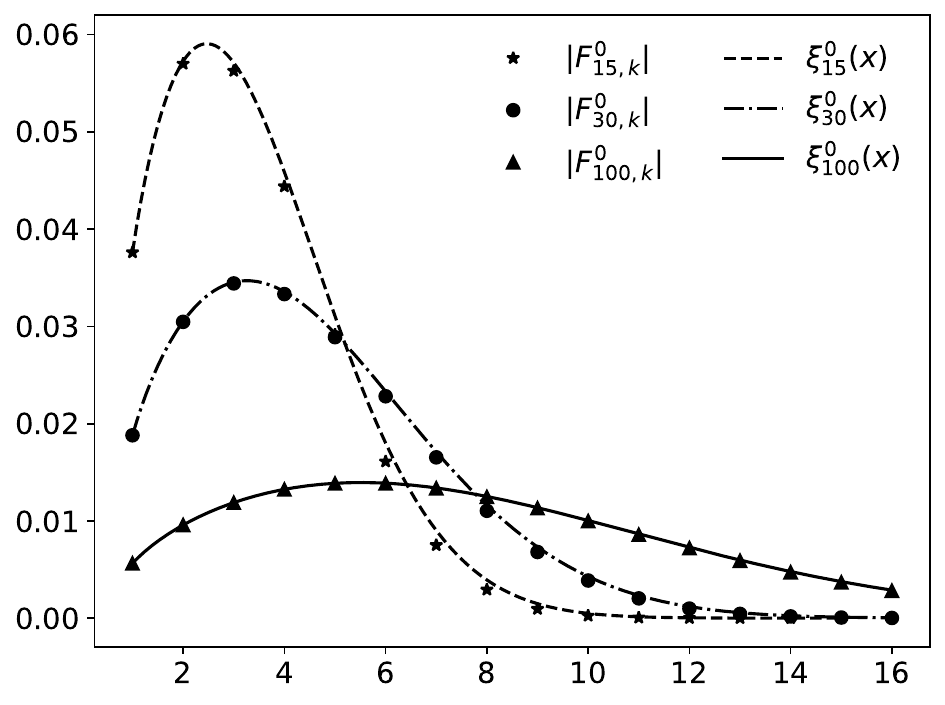}
    \end{minipage}\hfill
    \begin{minipage}{0.5\textwidth}
        \centering
        \includegraphics[width=\textwidth]{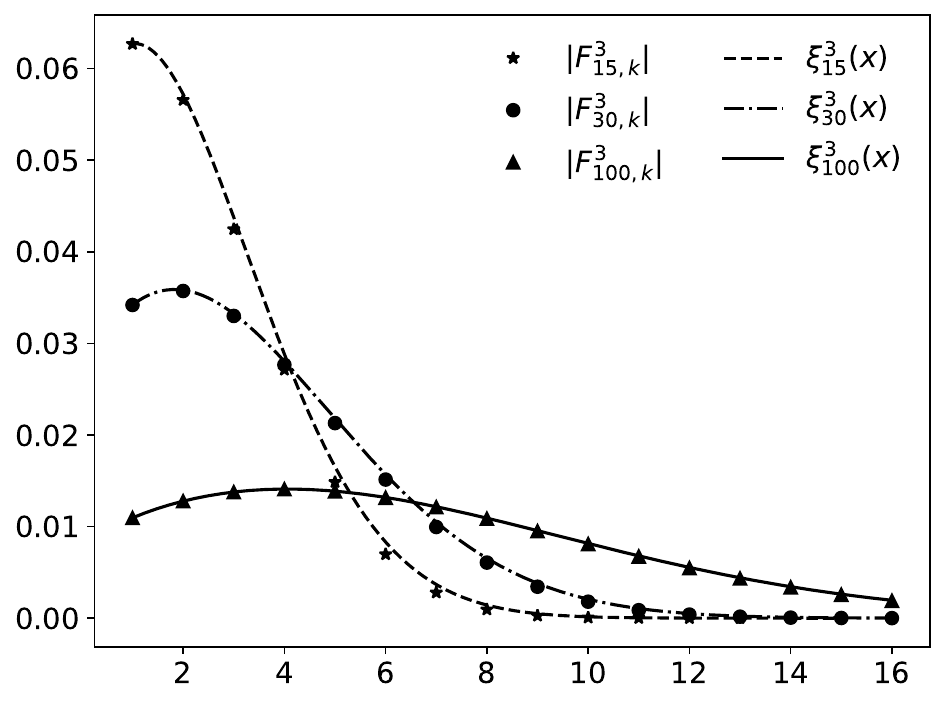}
    \end{minipage}
    \caption{Visual demonstration on the tightness of the upper bound \( \xi^\abs{j}_m(x) \) in \eqref{eq:asymp}.}
    \label{fiq:upper}
    \end{figure}

\section{Proof of Theorem~\ref{thm:main} for the unit disk}
\label{sec:schur}

    Before proceeding with the proof of Theorem~\ref{thm:main} in the case $\Omega = D$ is the unit disk, let us recall a classical tool for proving boundedness for infinite matrices with positive elements, namely the Schur test \cite{Schur1911}.
    \begin{theorem} \label{thm:help_Schur}
        Let $B=(b_{p,q})_{p,q = 1}^{\infty}$ be an infinite matrix with nonnegative elements $b_{p,q} \geq 0$ for all $p$ and $q$. 
        Suppose there are two positive infinite vectors $[u_l]^\infty_{l = 1}$ and $[v_l]^\infty_{l = 1}$ such that
        \begin{equation}\label{eq:schur}
            \sum^\infty_{q = 1} b_{p,q} u_q \leq C_1 v_p \quad \text{and} \quad \sum^\infty_{p = 1} b_{p,q} v_p \leq C_2 u_q \qquad \forall p, q \in \N,
        \end{equation}
        where \( C_1, C_2 > 0 \) are independent of the indices \( p \) and \( q \).
        Then $B\in \mathscr{L}(\ell^2(\N))$, with
        \begin{equation*}
            \|B\|_{\mathscr{L}(\ell^2(\N))} \leq \sqrt{C_1 C_2}.
        \end{equation*}
    \end{theorem}

    According to Lemma~\ref{lemma:norms}, the continuity of $F: L^2(D) \to \mathscr{L}_{\rm HS}(L^2_\diamond(\partial D))$ follows by showing that the infinite lower triangular matrices $F^{|j|}: \ell^2(\N) \to \ell^2(\N)$, defined by \eqref{eq:inf_F_matrix}, are uniformly bounded over $j \in \Z$. As a consequence, the following lemma completes the proof of Theorem~\ref{thm:main} for $\Omega = D$.
    \begin{lemma}\label{Lemma_F_L2}
        The family of infinite matrices $\{ F^\abs{j} \}_{j \in \Z}$ is uniformly bounded on $\ell^2(\N)$. More precisely,
        \begin{equation}
        \label{eq:F_norm}
            \big\| F^\abs{j} \big\|_{\mathscr{L}\left( \ell^2(\N) \right)} \leq  2^\frac{7}{2} 
        \end{equation}
        for all $j \in \Z$.
    \end{lemma}

    \begin{proof}
        Define a pair of infinite vectors via
        \begin{equation*}
            \vek{u} = \big[l^{-\frac{1}{2}} \big]_{l = 1}^{\infty}
            \qquad \text{and}
            \qquad \vek{v} = \big[ (l + \abs{j})^{-\frac{1}{2}} \big]_{l = 1}^{\infty}.
        \end{equation*}
        We aim to show that these satisfy the conditions \eqref{eq:schur} in Theorem~\ref{thm:help_Schur} for the matrix $|F^{|j|}| = (|F_{m,k}^{|j|}|)_{m,k=1}^\infty = (-F_{m,k}^{|j|})_{m,k=1}^\infty$ with the constants $C_1 = 4$ and $C_2 = 32$. 
        
        Let us start with the first inequality in \eqref{eq:schur}. 
        Recalling that $F^\abs{j}$ is lower triangular and resorting to the upper bound \eqref{eq:asymp}, we can estimate as follows:
             \begin{align} \label{eq:schur_first}
        \sum_{k = 1}^\infty \big| F^\abs{j}_{m, k} \big| u_k 
             &\leq \frac{1}{\sqrt{\pi}} \, \sum^m_{k = 1} \frac{\sqrt{2 k + \abs{j} - 1}}{m + \abs{j}} \, {\rm e}^{-\frac{2(k + |j|) (k - 1)}{2 m + \abs{j} + 1 }} \frac{1}{\sqrt{k}} \nonumber \\[1mm]
        &\leq \sqrt{\frac{2}{\pi}} \frac{1}{m + \abs{j}}\, \sum^m_{k = 1} \sqrt{\frac{k + \abs{j}}{k}} \, {\rm e}^{-\frac{2(k + |j|) (k - 1)}{2 m + \abs{j} + 1 }} \nonumber \\[1mm]
             &\leq \sqrt{\frac{2}{\pi}} \frac{1}{m + \abs{j}} \left( \int_{1}^m  \sqrt{\frac{x + \abs{j}}{x}} \, \e^{-\frac{2(x + \abs{j})(x - 1)}{2m + \abs{j} + 1 }} \di x  + \sqrt{\abs{j} + 1}  \right),
        \end{align}
        where the final step corresponds to the integral test to bound the sum over the index \( k \), with $\sqrt{\abs{j} + 1}$ being the value of the summand at $k=1$. Note that the integral test can be employed in its basic form since the summand is monotonically decreasing in $k$, as can be easily checked via differentiation. 
        
        Let us concentrate on the integral term on the right-hand side of \eqref{eq:schur_first}. Denote $a = 2m + \abs{j} + 1$ and make the change of variables 
        $$
        t^2 = \frac{2}{a} (x + \abs{j})(x - 1), \qquad \di x  = \frac{\sqrt{2  a  (x + |j|) (x-1)}}{2 x + \abs{j} - 1} \, \di t \leq 2 \, \sqrt{\frac{a \, x}{x + \abs{j}}} \, \di t
        $$ 
        for all $x \geq 1$, as can be verified through a straightforward calculation.
        This leads to 
        \begin{equation} \label{eq:errorf}
            \int^m_1 \sqrt{\frac{x + \abs{j}}{x}} \e^{-\frac{2 (x + \abs{j})(x - 1)}{a}} \, \di x
            \leq 2 \sqrt{a} \int^{\sqrt{\frac{2 (m + \abs{j} (m - 1))}{a}}}_0 \e^{-t^2} \, \di t
            \leq \sqrt{\pi a},
        \end{equation}
        where the last step follows by integrating up to infinity. 
        Combining \eqref{eq:schur_first} and \eqref{eq:errorf} finally gives
        \begin{equation}
        \label{eq:non_sharp1}
            \sum_{k = 1}^\infty \big| F^\abs{j}_{m, k} \big| u_k
            \leq \frac{\sqrt{2}}{m + \abs{j}} \left( \sqrt{2m + \abs{j} + 1} + \sqrt{\frac{\abs{j} + 1}{\pi}} \right)
            < 4 \, \frac{1}{\sqrt{m + \abs{j}}} = 4 \, v_m
        \end{equation}
        for all $m \in \N$ and $j \in \Z$. This proves the first part of \eqref{eq:schur}.

        We prove the second inequality in \eqref{eq:schur} in a similar manner, that is, we recall that $F^{|j|}$ is lower triangular, use the upper bound \eqref{eq:asymp}, and subsequently approximate the sum with an integral:
        \begin{align} \label{eq:schur_second}
            \sum^\infty_{m = 1} \big| F^\abs{j}_{m, k} \big| v_m 
                 &\leq  \frac{1}{\sqrt{\pi}} \sum^\infty_{m = k} \frac{\sqrt{2 k + \abs{j} - 1}}{m + \abs{j}} \, {\rm e}^{-\frac{2(k + |j|) (k - 1)}{2 m + \abs{j} + 1 }} \frac{1}{\sqrt{m + |j|}}
            \nonumber \\[1mm]
            &= \frac{1}{\sqrt{\pi}} \sqrt{2 k + \abs{j} - 1} \, \sum^\infty_{m = k} \left( \frac{1}{m + \abs{j}} \right)^{\frac{3}{2}} \e^{-\frac{2 (k + \abs{j})(k - 1)}{2 m + \abs{j} + 1}} \nonumber \\[1mm]
            &\leq \frac{3^{\frac{3}{2}} {\rm e}^2}{\sqrt{\pi}} \sqrt{2 k + \abs{j} - 1} \, \sum^\infty_{m = k} \left( \frac{1}{2m + \abs{j} + 1} \right)^{\frac{3}{2}} \e^{-\frac{2 k (k + \abs{j})}{2m + \abs{j} + 1}} \nonumber \\[1mm]
            &\leq \frac{3^{\frac{3}{2}} {\rm e}^2}{\sqrt{\pi}} \sqrt{2 k + \abs{j} - 1} \left( \int^\infty_k \left( \frac{1}{2x + \abs{j} + 1} \right)^{\frac{3}{2}} \e^{-\frac{b}{2x + \abs{j} + 1}} \, \di x + \left( \frac{3}{2  \e  b} \right)^{\frac{3}{2}} \right),
        \end{align}
        where $b = 2 k (k + \abs{j})$. The last inequality is obtained by observing that the summand is increasing as a function of $m$ on $(0, m^*)$ and decreasing on $(m^*, \infty)$, where $m^* = m^*(k, |j|)$ is the unique critical point of the summand, characterized by
        \begin{equation*}
            \frac{3}{2} (2 m^* + \abs{j} + 1) = b.
        \end{equation*}
        The additional term on the right-hand side of \eqref{eq:schur_second} is the maximal value of the summand, attained at $m^*$, the inclusion of which ensures the validity of the upper bound provided by the integral test.
        
        Making the change of variables 
        $$
        t^2 = \frac{b}{2 x + \abs{j} + 1}, \qquad \di x = - \frac{b}{t^3} \, \di t
        $$
        yields
        \begin{equation} \label{eq:schur_third}
            \int_{k}^{\infty} \left( \frac{1}{2x + \abs{j} + 1} \right)^{\frac{3}{2}} \e^{-\frac{b}{2x + \abs{j} + 1}} \, \di x 
            = \sqrt{\frac{1}{b}} \int_0^{\sqrt{\frac{b}{2k + \abs{j} + 1}}} \e^{-t^2} \, \di t
            \leq \frac{1}{2}\sqrt{\frac{\pi}{b}}, 
        \end{equation}
        where the inequality corresponds to integrating up to infinity.
        Substituting \eqref{eq:schur_third} in \eqref{eq:schur_second} and expanding \( b =  2 k (k + \abs{j})\), we get
        \begin{align}
        \label{eq:non_sharp2}
            \sum^\infty_{m = 1} \big| F^\abs{j}_{m, k} \big| v_m &\leq \frac{3^{\frac{3}{2}} \e^2}{\sqrt{\pi}} \sqrt{2 k + \abs{j} - 1} \left( \frac{1}{2} \sqrt{\frac{\pi}{2 k (k + \abs{j})}} + \left( \frac{3}{4 {\rm e} k (k + \abs{j})} \right)^{\frac{3}{2}} \right) \nonumber \\[1mm]
            &= \frac{3^\frac{3}{2} \e^2}{2} \sqrt{\frac{2k + \abs{j} - 1}{2 k (k + \abs{j})}} + \sqrt{\frac{\e}{\pi}} \left( \frac{9}{2} \right)^\frac{3}{2} \sqrt{\frac{2k + \abs{j} - 1}{(2 k (k + \abs{j}))^3}} \nonumber \\[1mm]
            &\leq \left( \frac{3^\frac{3}{2} \e^2}{2} + \sqrt{\frac{\e}{\pi}} \left( \frac{9}{2} \right)^\frac{3}{2} \right) \frac{1}{\sqrt{k}} \leq 32 \, u_k,
        \end{align}
        for all $k \in \N$ and $j \in \Z$. This proves the second part of \eqref{eq:schur}.

        The infinite matrix $|F^{|j|}|$ thus satisfies the conditions \eqref{eq:schur} in Theorem~\ref{thm:help_Schur} with $C_1 = 4$ and $C_2 = 32$. Consequently,
        \begin{equation*}
            \big\| F^\abs{j} \big\|_{\mathscr{L}(\ell^2(\N))}  = \big\| \, | F^\abs{j} | \,  \big\|_{\mathscr{L}(\ell^2(\N))} \leq \sqrt{C_1 C_2} = 2^\frac{7}{2} \qquad \forall j\in \Z,
        \end{equation*}
        which completes the proof.
    \end{proof}

\begin{remark}
The constant on the right-hand side of \eqref{eq:F_norm} is not optimal as,~e.g.,~the estimates \eqref{eq:non_sharp1} and \eqref{eq:non_sharp2} in the proof of Lemma~\ref{Lemma_F_L2} could be slightly sharpened. However, we consider presenting such a non-optimized bound well-motivated because, combined with Lemma~\ref{lemma:norms}, it reveals the approximate magnitude of the norm of $F: L^2(D) \to \mathscr{L}_{\rm HS}(L^2_\diamond(\partial D))$.
\end{remark}

\section{Generalization for \texorpdfstring{\( C^{1, \alpha} \)}{C1,a} domains and the proof of Corollary~\ref{cor:1}}
\label{sec:generalization}

Let $\Omega \subset \R^2$ be a simply-connected $C^{1, \alpha}$ domain and consider a conductivity perturbation $\eta \in L^2(\Omega)$. As in \cite[Section~2]{2d_paperi}, let $\Phi: D \to \Omega$ be a Riemann mapping, denote it inverse by $\Psi$, and define $\widetilde{\eta} = \eta \circ \Phi$. Due to the Kellogg--Warschawski theorem (see, e.g., \cite[Theorem 3.6 \& Exercise 3.3.5]{Pommerenke1992}), both $\Phi$ and $\Psi$ have extensions, with H\"older-continuous and non-vanishing complex derivatives $\Phi'$ and $\Psi'$,  to the closures of their respective domains. Based on \cite[eqs.~(2.2) \& (2.4)]{2d_paperi} and Lemma~\ref{lemma:norms} and \ref{Lemma_F_L2}, we have
\begin{align*}
\label{eq:proof_of_main}
\| F \eta \|_{\mathscr{L}_{\rm HS}( L^2_\diamond(\partial \Omega ))} &\leq \|\Phi' \|_{L^\infty(\partial D)} \| F \widetilde{\eta} \|_{\mathscr{L}_{\rm HS}( L^2_\diamond(\partial D))} \\[1mm]
& \leq \|\Phi' \|_{L^\infty(\partial D)} \| F \|_{\mathscr{L}(L^2(D),\mathscr{L}_{\rm HS}(L^2_\diamond(\partial D)))} \| \widetilde{\eta} \|_{L^2(D)} \\[1mm]
& \leq  16 \, \|\Psi' \|_{L^\infty(\partial \Omega)} \|\Phi' \|_{L^\infty( \partial D)}  \| \eta \|_{L^2(\Omega)},
\end{align*}
where we abused the notation by denoting the linearized forward map at the unit conductivity by $F$ for both $\Omega$ and $D$. Theorem~\ref{thm:main} has now been proved in its full extent. 

Let us then complete this paper by proving Corollary~\ref{cor:1}. The assertion follows immediately if one shows that the embedding 
\[
\mathcal{I}_\varepsilon: \mathscr{L}_{\rm HS}( L^2_\diamond(\partial \Omega)) \hookrightarrow\mathscr{L}_{\rm HS}( H^{\varepsilon}_\diamond(\partial \Omega), H^{-\varepsilon}_\diamond(\partial \Omega))
\]
is compact for any (small enough) $\varepsilon > 0$. In fact, we will prove this result for any simply-connected {\em Lipschitz domain} $\Omega$. 

To this end, let $\{ \phi_i \}_{i \in \N} \subset H^{1/2}_{\diamond}(\partial \Omega)$ be eigenfunctions of the compact self-adjoint operator $\Lambda(1): L^2_\diamond(\partial \Omega) \to L^2_\diamond(\partial \Omega)$ forming an orthonormal basis for $L^2_\diamond(\partial \Omega)$, and let $\{\lambda_i \}_{i \in \N} \subset \R_+$ be the corresponding eigenvalues that converge monotonically to zero as $i$ tends to infinity. The claimed smoothness of the eigenfunctions is a consequence of $\Lambda(1): H^{-1/2}_\diamond(\partial \Omega) \to H^{1/2}_\diamond(\partial \Omega)$ being a positive isomorphism; see,~e.g.,~\cite{Hyvonen18}. Recall that $\langle \, \cdot \, , \, \cdot \, \rangle_\varepsilon: H^\varepsilon(\partial \Omega) \times H^{-\varepsilon}(\partial \Omega) \to \C$ denotes the  sesquilinear dual evaluation between $H^\varepsilon(\partial \Omega)$ and $H^{-\varepsilon}(\partial \Omega)$. It follows from a simple extension of \cite[Lemma~1]{Hyvonen18} that
\begin{equation}
\label{eq:s-inner}
\langle g, h \rangle_{H^\varepsilon(\partial \Omega)} = \sum_{i \in \N} \lambda_i^{-2 \varepsilon} \langle g, \phi_i \rangle_\varepsilon \overline {\langle h, \phi_i \rangle_\varepsilon}, \qquad \varepsilon \in [-\tfrac{1}{2}, \tfrac{1}{2}],
\end{equation}
defines an inner product for $H^\varepsilon_\diamond(\partial \Omega)$, compatible with the standard topology of $H^\varepsilon_\diamond(\partial \Omega)$. A direct calculation verifies that the scaled eigenfunctions
\begin{equation}
\label{eq:s-basis}
\phi_i^\varepsilon = \lambda_i^\varepsilon \phi_i, \qquad i \in \N,
\end{equation}
form an orthonormal basis for $H^\varepsilon_\diamond(\partial \Omega)$, $\varepsilon\in [-\tfrac{1}{2}, \tfrac{1}{2}]$,  with respect to the inner product \eqref{eq:s-inner}. See also \cite[Appendix~B]{Garde2020}.

 Accordingly, an inner product for the Hilbert space $\mathscr{L}_{\rm HS}( H^{\varepsilon}_\diamond(\partial \Omega), H^{-\varepsilon}_\diamond(\partial \Omega))$ can be defined by
\[
\langle T_1, T_2 \rangle_{\mathscr{L}_{\rm HS}( H^{\varepsilon}_\diamond(\partial \Omega), H^{-\varepsilon}_\diamond(\partial \Omega))} = \sum_{p \in \N} \big \langle T_1 \phi_p^\varepsilon, T_2 \phi_p^\varepsilon \big \rangle_{H^{-\varepsilon}(\partial \Omega)}, \qquad \varepsilon  \in [-\tfrac{1}{2}, \tfrac{1}{2}],   
\]
for linear operators $T_1, T_2: H^{\varepsilon}_\diamond(\partial \Omega) \to H^{-\varepsilon}_\diamond(\partial \Omega)$ \cite{Weidmann1980}. Hence, it follows from \eqref{eq:s-inner} and \eqref{eq:s-basis} that the rank-one operators $\{ \phi_{i,j}^{\varepsilon, \otimes} \}_{i,j \in \N}$, defined via
\[
\phi_{i,j}^{\varepsilon, \otimes} : g \mapsto (\lambda_i \lambda_j)^{-\varepsilon}  \langle g, \phi_j \rangle_{\varepsilon} \, \phi_i,
\]
form an orthonormal basis of  $\mathscr{L}_{\rm HS}( H^{\varepsilon}_\diamond(\partial \Omega), H^{-\varepsilon}_\diamond(\partial \Omega))$ for any $\varepsilon \in [-\tfrac{1}{2}, \tfrac{1}{2}]$.

\begin{proposition}
    Let $\Omega \subset \R^2$ be a simply-connected Lipschitz domain. The embedding $\mathcal{I}_\varepsilon: \mathscr{L}_{\rm HS}( L^2_\diamond(\partial  \Omega)) \hookrightarrow\mathscr{L}_{\rm HS}( H^{\varepsilon}_\diamond(\partial \Omega), H^{-\varepsilon}_\diamond(\partial \Omega))$ is compact for any $\varepsilon \in (0, \tfrac{1}{2}]$.
\end{proposition}

\begin{proof}
     Let $\mathbb{M} = \{1, 2, \dots, M\}$. We introduce a sequence of finite-rank operators $\{ \mathcal{I}_\varepsilon^M \}_{M\in \N}$ via 
    \[
    \mathcal{I}_\varepsilon^M T = \sum_{(i,j) \in \mathbb{M}^2} \big\langle  T,  \phi_{i,j}^{\varepsilon, \otimes} \big\rangle_{\mathscr{L}_{\rm HS}( H^{\varepsilon}_\diamond(\partial \Omega), H^{-\varepsilon}_\diamond(\partial \Omega))} \, \phi_{i,j}^{\varepsilon, \otimes}, \qquad T \in \mathscr{L}_{\rm HS}( L^2_\diamond(\partial \Omega)),
    \]
     and demonstrate that it converges to $\mathcal{I}_\varepsilon$ in the operator norm as $M$ goes to infinity. This proves the assertion since the subspace of compact operators is closed in the operator topology for the bounded linear operators between the Banach spaces $\mathscr{L}_{\rm HS}( L^2_\diamond(\partial  \Omega))$ and $\mathscr{L}_{\rm HS}( H^{\varepsilon}_\diamond(\partial \Omega), H^{-\varepsilon}_\diamond(\partial \Omega))$.
     
      Define $\N_M^2  = \N^2 \setminus \mathbb{M}^2$. Because $\{  \phi_{i,j}^{\varepsilon, \otimes}\}_{i,j \in \N}$ is an orthonormal basis for $\mathscr{L}_{\rm HS}( H^{\varepsilon}_\diamond(\partial \Omega), H^{-\varepsilon}_\diamond(\partial \Omega))$, we have
    \begin{align*}
    \big\| (\mathcal{I}_\varepsilon - \mathcal{I}_\varepsilon^M) T \big\|_{\mathscr{L}_{\rm HS}( H^{\varepsilon}_\diamond(\partial \Omega), H^{-\varepsilon}_\diamond(\partial \Omega))}^2 &= \sum_{(i,j) \in \N_M^2} \big| \big\langle T,   \phi_{i,j}^{\varepsilon, \otimes} \big\rangle_{\mathscr{L}_{\rm HS}( H^{\varepsilon}_\diamond(\partial \Omega), H^{-\varepsilon}_\diamond(\partial \Omega))} \big|^2 \\[1mm]
    &= \sum_{(i,j) \in \N_M^2} (\lambda_i \lambda_j)^{-2\varepsilon} \,  \big| \big\langle T,   \phi_{i,j}^{0, \otimes} \big\rangle_{\mathscr{L}_{\rm HS}( H^{\varepsilon}_\diamond(\partial \Omega), H^{-\varepsilon}_\diamond(\partial \Omega))} \big|^2 \\[0mm]
    &=\sum_{(i,j) \in \N_M^2}  (\lambda_i \lambda_j)^{-2\varepsilon} \, \bigg| \sum_{p\in \N} \big \langle T \phi_p^\varepsilon, \phi_{i,j}^{0, \otimes} \phi_p^\varepsilon \big \rangle_{H^{-\varepsilon}(\partial \Omega)} \bigg|^2 \\[1mm]
    & = \sum_{(i,j) \in \N_M^2} \lambda_i^{-2\varepsilon} \, 
    \big| \big\langle T \phi_{j}^\varepsilon,  \phi_{i} \big \rangle_{H^{-\varepsilon}(\partial \Omega)} \big|^2 \\[1mm]
    & = \sum_{(i,j) \in \N_M^2} \lambda_i^{-2\varepsilon} \lambda_j^{2\varepsilon} \, \big| \big\langle T \phi_j,  \phi_{i} \big \rangle_{H^{-\varepsilon}(\partial \Omega)} \big|^2 \\[1mm]
     & = \sum_{(i,j) \in \N_M^2} (\lambda_i \lambda_j)^{2\varepsilon} \, \big| \big\langle T \phi_{j},  \phi_{i} \big \rangle_{L^{2}(\partial \Omega)}\big|^2 \\[1mm]
     & \leq \, (\lambda_1 \lambda_{M+1})^{2 \varepsilon}  \sum_{(i,j) \in \N_M^2} \big| \big\langle T \phi_{j},  \phi_{i} \big \rangle_{L^{2}(\partial \Omega)}\big|^2 \\[1mm] 
     & \leq  (\lambda_1 \lambda_{M+1})^{2 \varepsilon} \, \| T \|_{\mathscr{L}_{\rm HS}( L^2_\diamond(\partial \Omega))}^2.
    \end{align*}
    Since $T \in \mathscr{L}_{\rm HS}( L^2_\diamond(\partial \Omega))$ is arbitrary and $\lambda_{M+1} \to 0$ as $M \to \infty$, we conclude that $\mathcal{I}_\varepsilon^M$ converges to $\mathcal{I}_\varepsilon$ in the topology of
    \[
    \mathscr{L}\big(\mathscr{L}_{\rm HS}( L^2_\diamond(\partial \Omega)), \mathscr{L}_{\rm HS}( H^{\varepsilon}_\diamond(\partial \Omega), H^{-\varepsilon}_\diamond(\partial \Omega))\big), \qquad \varepsilon \in (0, \tfrac{1}{2}],
    \]
    as $M \to \infty$. This completes the proof.
\end{proof}

\subsection*{Acknowledgments}
This work was supported by the Academy of Finland (decisions 353081 and 358944).

%% ------------------------- sources ------------------------- %%
\phantomsection
\bibliographystyle{plain}
\bibliography{sources.bib}
%% ------------------------- sources ------------------------- %%

\end{document}